\documentclass{article}
\usepackage{amsmath,amsthm,amssymb}

\theoremstyle{theorem}
\newtheorem{theorem}{Theorem}
\newtheorem{proposition}{Proposition}

\newtheorem{lemma}{Lemma}
 
\theoremstyle{definition}

\newtheorem*{remark}{Remark}
\newtheorem*{remarks}{Remarks}
\newtheorem*{claim}{Claim}

\newtheorem{problem}{Problem}

\newcommand{\lr}[1]{\left\langle {#1} \right\rangle}
\newcommand{\cavg}[1]{\frac 1q \sum_{{#1}\in \mathbb{F}_q}}
 
\begin{document}

\title{Sums of Powers in Large Finite Fields: \\ A Mix of Methods}
\markright{Finite Field Method Mix}
\author{Vitaly Bergelson, Andrew Best, and Alex Iosevich}

\maketitle

\begin{abstract}
Can any element in a sufficiently large finite field be represented as a sum of two $d$th powers in the field? In this article, we recount some of the history of this problem, touching on cyclotomy, Fermat's last theorem, and diagonal equations. Then, we offer two proofs, one new and elementary, and the other more classical, based on Fourier analysis and an application of a nontrivial estimate from the theory of finite fields. In context and juxtaposition, each will have its merits.
\end{abstract}
\section{Introduction.} We denote by $\mathbb{F}_q$ the finite field with $q$ elements, where $q$ is a power of a prime.
Consider the following problem, which has been attacked several times in different centuries, and just two of its extensions.
\begin{problem}\label{two elts; no coeffs} Fix an integer $d > 1$. Show that for any sufficiently large finite field $\mathbb{F}_q$, we have
\begin{equation*} \mathbb{F}_q \ = \ \{ x^d + y^d : x, y \in \mathbb{F}_q\},
\end{equation*}
or, in other words, every element of $\mathbb{F}_q$ is a sum of two $d$th powers.
\end{problem}
\begin{problem}\label{two elts; with coeffs} Fix an integer $d > 1$. Show that for any sufficiently large finite field $\mathbb{F}_q$, we have for every $a,b \in \mathbb{F}_q^\times$ that
\begin{equation*} \mathbb{F}_q \ = \ \{ ax^d + by^d : x, y \in \mathbb{F}_q\}.
\end{equation*}
\end{problem}
\begin{problem} \label{count solutions} Fix integers $n > 1$ and $k_1,\ldots,k_n > 0$ . Given an element $b \in \mathbb{F}_q$ and coefficients $a_1, \ldots, a_n \in \mathbb{F}_q^\times$, determine the number of solutions $(x_1,\ldots, x_n) \in \mathbb{F}_q^n$ to the \emph{diagonal equation}
\begin{equation}\label{diag eqn} a_1x_1^{k_1} + \cdots + a_nx_n^{k_n} \ = \ b.
\end{equation}
\end{problem}
\begin{remarks} \begin{enumerate}
\item Looking at Problems~\ref{two elts; no coeffs}~and~\ref{two elts; with coeffs} with fresh eyes, one might first observe that since $x^{q-1} = 1$ for all $x \in \mathbb{F}_q^\times$, it follows that $\{x^{q-1} + y^{q-1} : x,y\in\mathbb{F}_q\} = \{0,1,2\}$, which is usually not all of $\mathbb{F}_q$. Thus, one should expect to require $q > d+1$, hence the ``sufficiently large'' in the formulation. As we will see, in the case $d = 2$, this requirement is unnecessary.
\item The progression in difficulty of the problems should be clear, with Problem~\ref{count solutions} a significant step up. Indeed, when $n=2$ and $k_1=k_2 = d$, knowing merely that the number of solutions to equation~\eqref{diag eqn} is positive for each choice of $a_1$, $a_2$, $b$ suffices to solve Problem~\ref{two elts; with coeffs}. However, even the more modest jump from Problem~\ref{two elts; no coeffs} to Problem~\ref{two elts; with coeffs} can pose a challenge in that some approaches that work for the former seem unable to be easily upgraded to work for the latter. 
\item According to Small, Kaplansky privately conjectured that the ``outrageous'' statement of Problem~\ref{two elts; no coeffs} holds \cite{small}. We will see what role the article \cite{small} occupies in the panoply of results we will survey shortly.
\end{enumerate}
\end{remarks}
Mathematicians who have been caught on paper being interested in diagonal equations include Lagrange, H. M. Weber,\footnote{Be careful with the initials. A contemporary of both the sociologist M. Weber and the physicist H. F. Weber, the latter of whose lectures' lack of Maxwell's equations spurred Einstein to spurn his mentorship, the mathematician H. M. Weber worked on algebra, analysis, and number theory, incorporating many of his results into his well-regarded textbook \textit{Lehrbuch der Algebra}. According to Schappacher in \cite{schap}, Weber and Dedekind ``[took] a decisive step towards the creation of modern algebraic geometry'' with their publication of \cite{dw}.} Cauchy, Skolem, Gauss, V. A. Lebesgue,\footnote{Not the famous analyst!} Dickson,\footnote{A prolific author, L. E. Dickson wrote (at least) two important books, namely the first comprehensive book on finite fields and the three-volume \textit{History of the Theory of Numbers}. In a book review, he wrote, ``Fricke's \textit{Algebra} is a worthy successor to Weber's \textit{Algebra}, which it henceforth displaces'' \cite{bookreview}.} Hurwitz,\footnote{Here we write ``Hurwitz'' to mean the well-known A. Hurwitz, who collaborated with his generally overshadowed older brother J. Hurwitz on complex continued fractions.} and Weil. 
In the process of proving his celebrated four-square theorem, Lagrange showed in \cite{lag} that, given any $b \in \mathbb{F}_p^\times$, we have $\mathbb{F}_p = \{ x^2 + by^2 : x,y \in \mathbb{F}_p \}$. Though not the first to do so, Weber solved Problem~\ref{two elts; no coeffs} in the case $d = 2$:
\begin{proposition}[{\cite[p. 309]{weber}}]\label{weberarg} Every element in a finite field $\mathbb{F}_q$ is the sum of two squares.
\end{proposition}
\begin{proof} Suppose $\mathbb{F}_q$ has characteristic 2, so that $q = 2^r$ for some positive integer $r$. Let $c \in \mathbb{F}_q$. Since $c^{q} =
 c$, it follows that $c = (c^{2^{r-1}})^2$, that is, $c$ is a square. Thus, the element $c = (c^{2^{r-1}})^2 + 0^2$ is a sum of two squares.

Now suppose $\mathbb{F}_q$ has characteristic $p > 2$, with $p$ prime. Let $c \in \mathbb{F}_q$. If $c$ is already a square, there is nothing to show, so suppose $c$ is a nonsquare. We now analyze two cases, depending on whether $-1$ is a square in $\mathbb{F}_q$. On the one hand, suppose $-1$ is a square in $\mathbb{F}_q$, so there is some element $g \in \mathbb{F}_q$ such that $g^2 = -1$. Then, the clever identity
\begin{equation*}
c\ = \ \left(c+\frac 14 \right)^2 + \left(g\left(c-\frac 14\right)\right)^2
\end{equation*}
expresses $c$ as a sum of two squares. Note that when $p = 3$, the symbol $\frac 14$ means $1$.

On the other hand, suppose $-1 = p-1$ is a nonsquare in $\mathbb{F}_q$ instead. We always know that $1$ is a square in $\mathbb{F}_q$. Now let's consider the prime subfield $\mathbb{F}_p \subseteq \mathbb{F}_q$, which is generated by $1$. Since $1$ is a square and $p-1$ is a nonsquare, it follows that there is a pair of numbers $a$ and $a+1$, both in $\mathbb{F}_p$, and both nonzero, such that $a$ is a square and $a+1$ is a nonsquare. Since $a$ is a square, there is a $g \in \mathbb{F}_q$ such that $g^2 = a$. Since $c$ is a nonsquare and so is $\frac{1}{a+1}$, it follows\footnote{A short argument establishes that the product of two nonsquares is a square. The map $\phi: \mathbb{F}_q^\times \to \mathbb{F}_q^\times$ given by $\phi(x)= x^2$ is a group homomorphism with kernel $\ker (\phi) = \{1,-1\}$, so $H:= \mathrm{im}(\phi)$ is an index 2 subgroup of $\mathbb{F}_q^\times$. In other words, $\mathbb{F}_q^\times / H$ is the group of two elements. As $c$ and $\frac{1}{a+1}$ are nonsquares, we have $cH = \frac{1}{a+1}H$, hence $\frac{c}{a+1}H = (cH)(\frac{1}{a+1}H) = (cH)^2 = H$, hence $\frac{c}{a+1} \in H$.} that $\frac{c}{a+1}$ is a square, so write $\frac{c}{a+1} = h^2$ for some $h \in \mathbb{F}_q$. We conclude that
\begin{equation*}
c = \frac{c}{a+1} (a + 1) = h^2 ( a + 1) = h^2 a + h^2 = (hg)^2 + h^2,
\end{equation*}
which shows that $c$ is a sum of two squares.
\end{proof}
Weber's proof amounts to cleverly producing, for each $c \in \mathbb{F}_q$, a solution to the diagonal equation $x^2 + y^2 = c$. It can be upgraded as follows to a solution to Problem~\ref{two elts; with coeffs} in the case $d=2$:
\begin{proposition} \label{prob 2 solution when d = 2} For any finite field $\mathbb{F}_q$, we have for every $a,b \in \mathbb{F}_q^\times$ that
\begin{equation*} \mathbb{F}_q \ = \ \{ ax^2 + by^2 : x,y\in \mathbb{F}_q\}.
\end{equation*}
\end{proposition}
\begin{proof} There are two cases. First, suppose that exactly one of $a$ and $b$ is a square, say $a$. If $c$ is a square, then $\frac ca$ is a square, and hence $ax^2 = c$ has a solution in $\mathbb{F}_q$, which implies that $c \in \{ax^2 + by^2 : x,y \in \mathbb{F}_q\}$, and if $c$ is a nonsquare, then similarly $bx^2 = c$ has a solution in $\mathbb{F}_q$, which implies the same. Second, suppose $a$ and $b$ are both squares or both nonsquares, so that $\frac ab$ is a square. Let $g \in \mathbb{F}_q^\times$ satisfy $g^2 = \frac ab$. By Proposition~\ref{weberarg}, there exist $x_0, y_0 \in \mathbb{F}_q$ such that $\frac ca = x_0^2 + y_0^2$. Set $y_1 = y_0 g$. Then
\begin{equation*} c \ = \ ax_0^2 + ay_0^2 \ = \ ax_0^2 + a \left(\frac{y_1}{g} \right)^2 \ = \ ax_0^2 + by_1^2. \qedhere
\end{equation*}
\end{proof}
There is a more straightforward argument than Weber's. Long before Weber's textbook was published, Cauchy \cite{cauchy} had solved Problem~\ref{two elts; with coeffs} when $d = 2$:
\begin{proof}[Second Proof of Proposition \ref{prob 2 solution when d = 2}] If $\mathbb{F}_q$ has characteristic 2, then, as observed before, every element is already a square, which quickly finishes the argument. Otherwise, fix $a,b \in \mathbb{F}_q^\times$ and let $c \in \mathbb{F}_q$. Since there are $\frac{q+1}{2}$ squares in $\mathbb{F}_q$ and the cardinality of a subset of a finite field is invariant under multiplication or addition by a fixed nonzero element of the field, the sets $\{ ax^2 : x \in \mathbb{F}_q\}$ and $\{c-by^2 : y \in \mathbb{F}_q\}$ both have cardinality $\frac{q+1}{2}$, i.e., occupy slightly more than half of the field $\mathbb{F}_q$. Thus these two sets have at least one common element $z$, which satisfies $z = ax^2 = c - by^2$ for some $x,y \in \mathbb{F}_q$. Hence $c = ax^2 + by^2$.
\end{proof}
\begin{remark} Cauchy's argument is more well known than Weber's, in part because it has been rediscovered more frequently. See, e.g., \cite[Lemma 1]{d3}, where it appears without citation, perhaps because it was well known by then.
\end{remark}
Cauchy's argument cannot be generalized to solve Problem~\ref{two elts; with coeffs} in general, not even when $d = 3$. For example, in $\mathbb{F}_7 = \{0,1,2,3,4,5,6\}$, the set of cubes is $\{x^3 : x \in \mathbb{F}_7\} = \{0,1,6\}$, which has less than half the cardinality of the field. Thus the pigeonhole principle used in the proof does not apply, as we cannot hope to intersect two modified sets of cubes. While we are here, we also observe that $\{x^3 + y^3 : x,y \in \mathbb{F}_7\} = \{0,1,2,5,6\} \neq \mathbb{F}_7$, so the failure of the argument is of course caused by $\mathbb{F}_7$ itself and not by a lack of ingenuity in modifying the technique in some way.

It turns out that, except when $q = 4$ and $q = 7$, every element of $\mathbb{F}_q$ can be written as a sum of two cubes, which solves Problem~\ref{two elts; no coeffs} in the case $d=3$. This precise result, proved in an elementary way in the twentieth century, is due to Skolem\footnote{Unlike these other people, who did not help found finitism, Skolem was more keen on mathematical logic. Several of his results, including this one, were independently rediscovered due to the general disconnection in research networks that prevailed in the years before the internet.} \cite{skolem} when $q$ is prime and to Singh \cite{singh} in general. Skolem and many others addressed Problem~\ref{two elts; with coeffs} in the case $d = 3$; see \cite[pp. 325--326]{LN} for sources.

Diverting our attention from successes of elementary methods for the moment, we return to the nineteenth century. Initiating the study of cyclotomy,\footnote{See \cite{storer} for a treatment of cyclotomy.} a web of problems that all involve roots of unity, Gauss, according to Weil, ``obtain[ed] the numbers of solutions for all congruences $ax^3 - by^3 \equiv 1 (\text{mod } p)$" for primes $p$ with $p \equiv 1 \, (\text{mod } 3)$ and ``[drew] attention himself to the elegance of his method, as well as to its wide scope'' \cite{weil}. Weil adds that V. A. Lebesgue (and, by implication, Gauss as well) was unable to bring these vaunted methods to bear on the full generality of Problem~\ref{count solutions}. We remark that V. A. Lebesgue and others had modest successes on specific cases with the same flavor as the quoted result by Gauss---restricted to $q$ prime, and so on. It was Kummer who more securely connected cyclotomy to Gauss sums \cite{kummer1, kummer2, kummer3}, paving the way for further advancements in the theory of diagonal equations.

We enter the twentieth century again, this time wearing our cyclotomy goggles. From this viewpoint, Dickson was quite important. The following theorem of his, from two of dozens of his cyclotomy papers, illustrates another historical reason for interest in diagonal equations:
\begin{theorem}[\cite{d2,d1}]\label{dickson thm} Fix an odd prime $e > 2$. Then for all sufficiently large primes $p$, the equation
\begin{equation}\label{dickson eqn} x^e + y^e + z^e \ = \ 0 \, (\emph{mod } p) 
\end{equation}
has a solution $(x,y,z) \in (\mathbb{F}_p^\times)^3$.
\end{theorem}
\begin{remark} This theorem's \emph{raison d'\^{e}tre} was that it nullified an approach to proving the case of Fermat's last theorem with prime exponents. If the conclusion of Theorem~\ref{dickson thm} were false for, say, $e = 3$, then there would exist a sequence of primes $(p_n)$ tending to infinity for which the only solutions $(x,y,z) \in \mathbb{F}_{p_n}^3$ to equation~\eqref{dickson eqn} with $e =3$ would have at least one of $x, y, z$ divisible by $p_n$. It follows that if we had a solution in integers to the equation $x^3 + y^3 + z^3 = 0$, taking this equation modulo a large enough $p_n$ would yield a contradiction. Perhaps surprisingly, this observation appears in \cite{hurwitz}, but not \cite{d2} or \cite{d1}.
\end{remark}
As stated, Theorem \ref{dickson thm} is a partial solution to Problem~\ref{count solutions} but not to Problem~\ref{two elts; with coeffs} or even Problem~\ref{two elts; no coeffs}. Hurwitz obtained an improvement which pertains to Problem~\ref{two elts; with coeffs}:
\begin{theorem}[\cite{hurwitz}]\label{hurwitz thm} Fix an odd prime $e > 2$ and nonzero integers $a,b,c$. Then for all sufficiently large primes $p$, the equation
\begin{equation}\label{hurwitz eqn} ax^e + by^e + cz^e \ = \ 0 \, (\emph{mod } p) 
\end{equation}
has a solution $(x,y,z) \in (\mathbb{F}_p^\times)^3$.
\end{theorem}
\begin{remark} If $p$ is a prime such that $(x,y,z) \in (\mathbb{F}_p^\times)^3$ is a solution to equation~\eqref{hurwitz eqn}, then after subtracting $cz^e$ from both sides and dividing by $z^e$, we have
\begin{equation*}
a \left( \frac{x}{z} \right)^e + b \left( \frac{y}{z} \right)^e \ = \ -c,
\end{equation*}
which expresses an arbitrary nonzero element $-c$ of $\mathbb{F}_p$ as a member of the set $\{ ax^e + by^e : x,y \in \mathbb{F}_p \}$. Thus Theorem~\ref{hurwitz thm} solves Problem~\ref{two elts; with coeffs} when $q$ is prime and $d$ is an odd prime. We will make use of this division trick again.
\end{remark}

Having explained some partial solutions to Problem~\ref{count solutions}, and their connection to Problems~\ref{two elts; no coeffs}~and~\ref{two elts; with coeffs} in the case $q$ prime, it is reasonable to say that everything has been set in motion. In a few decades, some solutions to Problem~\ref{count solutions} have emerged. In this article we are not interested in the exact details of these solutions, but for completeness's sake, we mention that, broadly speaking, the approaches go either by Gauss sums, or by Jacobi sums, or by elementary methods.\footnote{Each of these methods, including a combined treatment of the elementary approaches by Stepanov \cite{stepanov} and Schmidt \cite{schmidt}, is treated in, e.g., \cite{LN}.}

Weil summarized historical progress on Problem~\ref{count solutions} up to 1949 in \cite{weil}, in the process simplifying and advancing some earlier work by Hasse and Davenport in \cite{dh}. It should be mentioned that the new results from \cite{weil} were discovered during the proofing process to be, in Weil's words, ``substantially identical'' to those of Hua and Vandiver in \cite{hv1}. What appears in \cite{weil} and \cite{hv1} has been made more elementary in several presentations, such as \cite{IR}, and can now be readily studied therefrom.

Small explicitly observed in the short article \cite{small} the fact that a certain estimate on the number of solutions to diagonal equations yields a solution to Problem~\ref{two elts; no coeffs}. For the sake of review, we now reproduce the one theorem from \cite{small}, with minor changes in notation for consistency: 
\begin{theorem}[\cite{small}]\label{small thm} Let $d$ be a positive integer, let $\mathbb{F}_q$ be a finite field, and put $\delta = \gcd (q-1,d)$. Assume $q > (\delta - 1)^4$. Then every element of $\mathbb{F}_q$ is a sum of two $d$th powers. (In particular, the conclusion holds if $q > (d-1)^4$, since $d \geq \delta$.)
\end{theorem}
\begin{proof} For $b \in \mathbb{F}_q$ let $N(b)$ denote the number of solutions $(x,y) \in \mathbb{F}_q \times \mathbb{F}_q$ of $x^d + y^d = b$. Then, by definition, $N(b) \geq 0$; we have to show that $q > (\delta - 1)^4$ implies $N(b) > 0$. We may assume $b \neq 0$, since 0 is certainly a sum of two $d$th powers. Then, by \cite[Corollary 1, p. 57]{joly}, we have $|N(b)-q| \leq (\delta-1)^2\sqrt q$. In particular, $N(b) - q \geq -(\delta-1)^2\sqrt q$, so that $N(b) \geq \sqrt q ( \sqrt q - (\delta - 1)^2)$. Hence $N(b) > 0$, for all $b$, provided $\sqrt q > (\delta-1)^2$, or in other words $q > (\delta - 1)^4$.
\end{proof}
\begin{remark}
If $d > 1$ is a fixed integer, one might wonder what is the largest integer $q$ for which not every element of $\mathbb{F}_q$ is the sum of two $d$th powers. By Theorem~\ref{small thm}, this value must be less than $(d-1)^4+1$. This problem might be interesting to pursue, in part because it is related to Waring's problem over finite fields.
\end{remark}

Small cites Jean-Ren\'e Joly's self-contained survey \cite{joly} about equations and algebraic varieties over finite fields. In the chapter endnotes, Joly states that the particular result Small would later quote is due independently to Davenport and Hasse in a 1934 paper \cite{dh} and to Hua and Vandiver in 1949 papers \cite{hv1} and \cite{hv2}. But one of the Hua--Vandiver papers \cite{hv2} cites \cite{dh} regarding the result in question!\footnote{It is possible that Hua and Vandiver learned about the Davenport--Hasse result after the publication of \cite{hv1} in 1949 and before the publication of \cite{hv2}\ldots in 1949. What a year 1949 was!} Anyway, as for the result Small uses, it follows from the development of Gauss and Jacobi sums over a few chapters in Joly.\footnote{The treatments in Joly \cite{joly} and Ireland--Rosen \cite{IR} are roughly equally accessible. After reading this paragraph and the one about Weil, the reader hopefully has some feeling of a frenzy, whether synchronous or asynchronous, around this topic.}

\begin{remark} We remark that although Small recorded a solution to Problem~\ref{two elts; no coeffs}, the inequality he quoted is general enough that he had essentially recorded a solution to Problem~\ref{two elts; with coeffs} as well.
\end{remark}
In the remaining two sections, we share two other ways to attack Problem~\ref{two elts; with coeffs}. The first is both new and elementary in the sense that it does not depend on Gauss sums, Jacobi sums, or hard counting of solutions to diagonal equations. As we will see, it is a soft averaging argument---what we are averaging will become apparent---and its main tool is the Cauchy--Schwarz inequality. The second way, which is not new, could accidentally be considered elementary if the reader blinks at the right moment. To be only a little more precise, we will reframe the situation using Fourier analysis, reducing the problem to a state where a single bolt of lightning from the Riemann hypothesis over finite fields creates a piece of fulgurite\footnote{Fulgurite is a general term for a mineral-like clump of dirt that can form where lightning strikes the ground.} to add to one's collection. We include this second approach as a byway to reiterate the continual importance of Fourier analysis and to emphasize a certain application from the high-minded theory of the Riemann hypothesis over finite fields, which we do not intend to explain here. We hope this will prove a useful juxtaposition with our elementary first proof and the arguments in this introduction.

\section{An elementary approach.}\label{sec 2}
From this section onward, we let $\mu_q$ denote the counting measure on $\mathbb{F}_q$, normalized to be a probability measure; thus, for any subset $A$ of $\mathbb{F}_q$, $\mu_q(A)$ equals $|A|/q$, the cardinality of $A$ divided by $q$. In context, $q$ will be fixed, so we will suppress the subscript and just write $\mu$. Besides measuring sets, we will also shift them: If $y$ lives in $\mathbb{F}_q$, denote by $A + y$ the set
\begin{equation*}
A + y \ := \ \{ x + y \in \mathbb{F}_q : x \in A\}.
\end{equation*}
We need a basic lemma.
\begin{lemma}\label{bezoutarg} Fix positive integers $d$ and $q$, and let $\delta = \gcd (d,q-1)$. Then
\begin{equation*} \{x^{\delta} : x \in \mathbb{F}_{q} \} \ = \ \{x^d : x \in \mathbb{F}_{q}\}.
\end{equation*}
\end{lemma}
\begin{proof} By B\'ezout's lemma, there exist integers $r$ and $s$ such that $rd + s(q-1) = \delta$. Hence, for all $x \in \mathbb{F}_{q}^\times$, we have
\begin{equation}\label{bezout manip} x^{\delta} \ = \ (x^r)^d(x^s)^{q-1} \ = \ (x^r)^d,
\end{equation}
since $y^{q-1} = 1$ for all $y \in \mathbb{F}_{q}^\times$. But equation \eqref{bezout manip} implies the set containment
\begin{equation*} \{x^{\delta} : x \in \mathbb{F}_{q} \} \ \subseteq \ \{x^d : x \in \mathbb{F}_{q}\}.
\end{equation*}
The reverse set containment holds since $x^d = (x^{d/\delta})^{\delta}$ for all $x \in \mathbb{F}_{q}$.
\end{proof}
Now, as promised in the introduction, we provide here an elementary proof of the following:
\begin{theorem}\label{elem thm} Fix an integer $d > 1$. Then, for every sufficiently large finite field $\mathbb{F}_q$ with characteristic larger than $d$ and every $a,b \in \mathbb{F}_q^\times$, we have
\begin{equation*} \mathbb{F}_q \ = \ \{ax^d + by^d : x,y \in \mathbb{F}_q\}.
\end{equation*}
\end{theorem}
\begin{remark} If we restrict to the special case where $q$ is prime, we can remove the assumption that $\mathrm{char}(\mathbb{F}_q) > d$. Indeed, let $q = p$ be prime, let $d > 1$, and set $\delta = \gcd(d,p-1)$. By Lemma~\ref{bezoutarg}, it follows that
\begin{equation*} \{ax^d + by^d : x,y \in \mathbb{F}_{p}\} \ = \ \{ax^{\delta} + by^{\delta} : x,y \in \mathbb{F}_{p}\}
\end{equation*}
for all $a, b \in \mathbb{F}_p^\times$. The result follows from the theorem since $\delta < p = \mathrm{char}(\mathbb{F}_{p})$. 
\end{remark}
 To prove Theorem~\ref{elem thm}, we will first need to state two lemmas. We prove the easier one immediately and defer the proof of the other one to the end of the section.
\begin{lemma}\label{dth power map has large image} Fix a finite field $\mathbb{F}_q$ and an integer $d > 1$, and let $A = \{ x^d : x \in \mathbb{F}_q \}$. Then $\mu(A) > \frac 1d$.
\end{lemma}
\begin{proof} The map $\phi : \mathbb{F}_q^\times \to \mathbb{F}_q^\times$ given by $\phi(x) = x^d$ is a group homomorphism. Since the equation $x^d = 1$ has at most $d$ distinct solutions in $\mathbb{F}_q$, it follows that $|\ker(\phi) |\leq d$. Thus
\begin{equation*} |A| - 1 \ = \ |\mathrm{im} (\phi )| \ = \ \frac{|\mathbb{F}_q^\times|}{|\ker (\phi ) |} \ = \ \frac{q - 1}{|\ker (\phi ) |} \ \geq \ \frac{q-1}{d},
\end{equation*}
which implies that
\begin{equation*} |A| \ \geq \  1 + \frac{q-1}{d} \ > \ \frac 1d + \frac{q-1}{d} \ = \ \frac qd.
\end{equation*}
Thus $\mu(A) = \frac{|A|}{q} > \frac 1d$. 
\end{proof}
\begin{lemma} \label{eqd} There is a positive function $\mathcal{E}(q,d) : \mathbb{N}^2 \to \mathbb{R}$ with these properties:
\begin{enumerate}
\item For any $q$, any subsets $A_q, B_q \subseteq \mathbb{F}_q$, and any polynomial $P \in \mathbb{F}_q[x]$ with degree $d$ satisfying $1 < d < \mathrm{char}(\mathbb{F}_q)$, we have
\begin{equation}\label{eqdbound} \left\vert \frac 1q \sum_{g \in \mathbb{F}_q} \mu(A_q \cap (B_q + P(g))) - \mu(A_q)\mu(B_q) \right| \ \leq \ \mathcal{E}(q,d).
\end{equation}
\item For any fixed $d$, we have $\lim_{q\to\infty} \mathcal{E}(q,d) \ = \ 0$. 
\end{enumerate}
\end{lemma}
\begin{remarks} \begin{enumerate}
\item In the first statement, we suppose $d > 1$ because Problems~\ref{two elts; no coeffs}~and~\ref{two elts; with coeffs} are trivially solved without this lemma when $d = 1$. In fact, for polynomials $P$ of degree 1, the left-hand side of \eqref{eqdbound} is always 0. On the other hand, it is not possible to remove the assumption that $d < \mathrm{char}(\mathbb{F}_q)$. This is because for certain choices of $P$, for example $P(x) = x^{\mathrm{char}(\mathbb{F}_q)} + x$, the supremum over $A_q$ and $B_q$ of the left-hand side of \eqref{eqdbound} is bounded away from 0, independently of $q$.
\item This lemma asserts that, eventually, the quantities $\mu(A_q \cap (B_q + P(g)))$ are of size $\mu(A_q)\mu(B_q)$ on average. We will see in the proof of Theorem~\ref{elem thm} that, when concrete choices of $A_q$ and $B_q$ are made, this currently vague statement will suddenly reveal useful information, namely the nontrivial intersection of $A_q$ and $B_q + P(g)$ for some $g \neq 0$. 
\item Our choice of $\mathcal{E}(q,d)$ will not be optimal in any reasonable sense.
\end{enumerate}
\end{remarks}

\begin{proof}[Proof of Theorem~\ref{elem thm}] We need only to show that $\mathbb{F}_q \ \subseteq \ \{ax^d + by^d : x,y \in \mathbb{F}_q\}$, as the reverse inclusion is trivial.
Let $A_q=\{ax^d:x\in\mathbb{F}_q\}$ and $B_q = \{-by^d : y \in \mathbb{F}_q \}$. Fix $c \in \mathbb{F}_q^\times$, and let $P(x) = cx^d$. Let $\mathcal{E}(q,d)$ be the function from Lemma~\ref{eqd}.
\begin{claim} For sufficiently large $q$, we have
\begin{equation*} \mu(A_q)\mu(B_q) - \frac 1q \mu(A_q \cap B_q) \ > \ \mathcal{E}(q,d).
\end{equation*}
\end{claim}
Indeed, by Lemma~\ref{dth power map has large image}, $\mu(A_q) > \frac 1d$ and $\mu(B_q) > \frac 1d$. Since $\lim_{q\to\infty} \mathcal{E}(q,d) = 0$ and $\frac{1}{q} \mu(A_q \cap B_q) \leq \frac 1q$, it follows that for sufficiently large $q$ we have
\begin{equation*}\label{figure it out} \mathcal{E}(q,d) + \frac 1q \mu(A_q \cap B_q) \ < \ \frac{1}{d^2} \ < \ \mu(A_q)\mu(B_q),
\end{equation*}
whence the claim.

Now, define the closely related quantities
\begin{align*}
S \ & = \ \frac 1q \sum_{g \in \mathbb{F}_q} \mu(A_q \cap (B_q + P(g))) \text{ and } \\
T \ & = \ \frac 1q \sum_{g \in \mathbb{F}_q^\times} \mu(A_q \cap (B_q + P(g)))
\end{align*}
We note that $S = T + \frac{1}{q} \mu(A_q \cap B_q)$, which follows by adding the missing $g = 0$ term to $T$.

For sufficiently large $q$, Lemma~\ref{eqd} implies that
\begin{equation} \label{above the floor} S - \mu(A_q)\mu(B_q) \ \geq \ - \mathcal{E}(q,d).
\end{equation}
Thus, for sufficiently large $q$, it follows by the claim and \eqref{above the floor} that
\begin{align*} T \ & = \ S - \frac 1q \mu(A_q \cap B_q) \\
& = \ \bigg( S - \mu(A_q)\mu(B_q)\bigg) + \left( \mu(A_q)\mu(B_q)-\frac 1q\mu(A_q\cap B_q) \right) \\
& > \ - \mathcal{E}(q,d) + \mathcal{E}(q,d) \ = \ 0.
\end{align*}
Since $T > 0$ and $T$ is a sum of nonnegative numbers, at least one summand is positive. Thus, there is an element $g \in \mathbb{F}_q^\times$ such that $\mu(A_q \cap(B_q + P(g))) > 0$. Having positive measure, the set $A_q \cap (B_q + P(g))$ must therefore be nonempty; hence there exist $x_1,x_2 \in \mathbb{F}_q$ such that $ax_1^d = -bx_2^d + cg^d$. Since $g \neq 0$, we can rearrange\footnote{We have encountered this trick before. In the introduction, it appeared in the remark that explains how Theorem~\ref{hurwitz thm} solves Problem~\ref{two elts; with coeffs} when $q$ is prime.} this equation to yield
\begin{equation*} c \ = \ a \left( \frac{x_1}{g} \right)^d + b \left( \frac{x_2}{g}\right)^d,
\end{equation*}
which shows that $c \in \{ ax^d+by^d : x,y \in \mathbb{F}_q \}$. But $c \in \mathbb{F}_q^\times$ was arbitrary, and of course $0 = a \cdot 0^d + b \cdot 0^d$; hence we are done.
\end{proof}
To prove the lemma that has done all the heavy lifting for us, we need some elementary facts about inner products. Fix a finite field $\mathbb{F}_q$. Define the inner product
\begin{equation*} \lr{f_1,f_2} \ := \ \int_{\mathbb{F}_q} f_1\overline{f_2} \ d\mu \ = \ \frac 1q \sum_{g \in \mathbb{F}_q} f_1(g)\overline{f_2(g)}
\end{equation*}
for functions $f_1,f_2 : \mathbb{F}_q \to \mathbb{C}$ and write the corresponding norm $|| f_1 || := \sqrt{\lr{f_1,f_1}}$. In this notation, the Cauchy--Schwarz inequality has a simple form:
\begin{equation} \label{cs} \left| \lr{f_1,f_2}\right| \leq ||f_1||\cdot ||f_2||.
\end{equation}
Indeed, for $a_1, \ldots, a_q, b_1,\ldots, b_q \in \mathbb{C}$, the Cauchy--Schwarz inequality states that 
\begin{equation}\label{csorig} \left\vert \sum_{i=1}^q a_i \overline{b_i} \right\vert^2 \ \leq \ \sum_{j=1}^q |a_j|^2 \sum_{k=1}^q|b_k|^2. \end{equation}
Setting\footnote{We are effectively enumerating the elements of $\mathbb{F}_q$ as $g_1, \ldots, g_q$ and writing $f_1(j)$ instead of $f_1(g_j)$. This abuse of notation is committed only to clarify how one form of Cauchy--Schwarz follows from another.} $a_j = f_1(j)$ and $b_k = f_2(k)$ and dividing both sides by $q^2$, we obtain
\[ \left\vert \frac{1}{q} \sum_{i=1}^q f_1(i)\overline{f_2(i)} \right\vert^2 \ \leq \ \left( \frac{1}{q} \sum_{j=1}^{q} |f_1(j)|^2\right) \left( \frac{1}{q} \sum_{k=1}^{q} |f_2(k)|^2\right),\]
which is written compactly as $|\lr{f_1,f_2}|^2 \ \leq \ \lr{f_1,f_1}\lr{f_2,f_2}$, whence \eqref{cs}.

All the functions we will take inner products with will be real-valued, so we will treat this inner product as a pleasantly linear tool. In particular, we will pull finite sums in and out, without warning.
\begin{proof}[Proof of Lemma \ref{eqd}]
Fix $q$ and subsets $A,B \subseteq \mathbb{F}_q$. In the proof of Theorem~\ref{elem thm} we only used polynomials of the form $P(x) = cx^d$ for $c \in \mathbb{F}_q$ nonzero, so we will stick with these for illustration.

If $1_A$ is the characteristic function of $A$, let $a_g$ be the function $\mathbb{F}_q \to \mathbb C$ defined by $x \mapsto 1_A(x+g) - \mu(A)$. We make two observations about these $a_g$'s, both involving shifts, and both justifying the utility of this weird function. First, after a change of variable $x \mapsto x - g'$, we see
\begin{align} 
\lr{a_g,a_{g'}} \ & = \ \int_{\mathbb{F}_q} \Big( 1_A(x+g) - \mu(A)\Big) \Big(1_A(x+g') - \mu(A) \Big) \ d\mu(x) \nonumber \\
& = \ \int_{\mathbb{F}_q}  \Big( 1_A(x+g-g') - \mu(A)\Big) \Big(1_A(x) - \mu(A) \Big) \ d\mu(x) \nonumber \\
& = \ \lr{a_{g-g'},a_0}. \label{first ag fact}
\end{align}
Second, one checks\footnote{Equation \eqref{dumb} holds since $x+y \in A$ if and only if $x \in A - y$, and \eqref{dumb2} since $1_A(x)1_B(x) = 1_{A \cap B}(x)$.} that
\begin{align}
\lr{a_{P(g)},1_B} \ & = \ \int_{\mathbb{F}_q} \Big( 1_A(x+P(g)) - \mu(A) \Big)1_B(x) \ d\mu(x) \nonumber \\
& = \ \int_{\mathbb{F}_q} \Big( 1_{A-P(g)}(x)-\mu(A) \Big)1_B(x) \ d\mu(x) \label{dumb} \\
& = \ \int_{\mathbb{F}_q} 1_{A-P(g)}(x)1_B(x) - \mu(A)1_B(x) \ d\mu(x) \nonumber\\
& = \ \int_{\mathbb{F}_q} 1_{(A-P(g)) \cap B}(x) - \mu(A)1_B(x) \ d\mu(x) \label{dumb2} \\
& = \ \mu((A-P(g))\cap B) - \mu (A)\mu (B), \nonumber
\end{align}
so that
\begin{equation}\label{second fact}
\lr{a_{P(g)},1_B} = \ \mu(A \cap (B+P(g))) - \mu (A)\mu (B).
\end{equation}

With the help of \eqref{second fact}, we begin. By Cauchy--Schwarz, we have
\begin{align}
\left\vert\cavg{g} \mu (A \cap (B+P(g))) - \mu(A)\mu(B) \right\vert \ & = \ \left\vert \lr{\cavg{g} a_{P(g)}, 1_{B}} \right\vert \nonumber     \\
& \leq \ \left\vert\left\vert \cavg{g} a_{P(g)} \right\vert\right\vert \cdot ||1_{B}|| \nonumber \\
& \leq \ \left\vert\left\vert \cavg{g} a_{P(g)} \right\vert\right\vert \label{test 42},
\end{align}
where we used the fact that $||1_B|| = \sqrt{\lr{1_B,1_B}} = \sqrt{\mu(B)} \leq 1$.

We prepare to bound the norm in \eqref{test 42} using the following elementary observation. Namely, if $h \in \mathbb{F}_q$ is fixed, the polynomial
\begin{equation*} P(x+h)-P(x) \ = \ c(x+h)^d - cx^d
\end{equation*}
has degree at most $d-1$ in $x$. Usually the degree will be $d-1$, but it can be smaller if $h = 0$ or if we did not suppose the characteristic of $\mathbb{F}_q$ is larger than $d$. This \emph{differencing} procedure allows us to reduce the degree of a polynomial whenever we can introduce this difference, which we will do aggressively.\footnote{Arguably the most notable use of this idea is in \cite{weyl}, wherein Weyl shows that if $Q(x)$ is a real polynomial with at least one irrational coefficient, then $\{ Q(n) : n \in \mathbb{N} \}$ is uniformly distributed modulo 1.}
If we are given parameters $h_1, h_2, \ldots, h_{d-1} \in \mathbb{F}_q$, let
\begin{align*} 
 P(x;h_1) \ & := \ P(x+h_1)-P(x), \\
 P(x;h_1,h_2) \ & := \ P(x+h_2;h_1)-P(x;h_1), \\
 & \; \vdots\\
 P(x;h_1,h_2, \ldots, h_{d-1}) & := \ P(x+h_{d-1}; h_1,h_2,\ldots, h_{d-2}) - P(x; h_1,h_2,\ldots, h_{d-2})
\end{align*}
be, respectively, the degree at most $d-1$, degree at most $d-2$, \ldots, and degree at most $1$ polynomials in $x$ obtained by consecutively differencing $P$ by $h_1$, then the result by $h_2$, and so on down to a (usually) linear polynomial. Now we return to bound the norm in \eqref{test 42}. This is a key step.

Observe, after using the first fact \eqref{first ag fact} about the $a_g$'s, reindexing a sum to introduce differencing, and applying the triangle and Cauchy--Schwarz inequalities,
\begin{align}
\left\vert\left\vert \cavg{g} a_{P(g)} \right\vert\right\vert^2 \ & = \ \lr{\cavg{g'} a_{P(g')},\cavg{g} a_{P(g)}} &&\text{by definition} \nonumber \\
& = \ \cavg{g'} \cavg{g} \lr{a_{P(g')},a_{P(g)}} \nonumber \\
& = \ \cavg{g'} \cavg{g} \lr{a_{P(g')-P(g)},a_0} && \text{by \eqref{first ag fact}} \nonumber \\
& = \ \cavg{h_1} \cavg{g} \lr{a_{P(g+h_1)-P(g)},a_0} \nonumber \\
& = \ \cavg{h_1} \lr{\cavg{g} a_{P(g;h_1)},a_0} \nonumber \\
& \leq \ \cavg{h_1} \left\vert\left\vert \cavg{g} a_{P(g;h_1)} \right\vert\right\vert. \label{vdc reduce}
\end{align}
Note that the last step follows\footnote{Note that $||a_0|| = \sqrt{\mu(A)-\mu(A)^2} \leq \frac 12$ since the function $x \mapsto x-x^2$ has maximum $\frac 14$ on $[0,1]$, but this better bound is not necessary for our argument.} since
\begin{equation*}
||a_0|| \ = \ \left\vert\left\vert 1_{A}-\mu(A)\right\vert\right\vert  \ = \ \sqrt{\mu(A) -\mu(A)^2}  \ \leq \ 1.
\end{equation*}
This argument in \eqref{vdc reduce} has reduced the degree of $P(g)$ by (at least) one. Encouraged, we prepare to iterate. For each $h_1$, making the displayed argument with $a_{P(g;h_1)}$ replacing $a_{P(g)}$ on the left-hand side, we get a bound on $\left\vert\left\vert\cavg{g} a_{P(g;h_1)}\right\vert\right\vert$ in terms of a new parameter $h_2$, which plays the same role as $h_1$ in the displayed argument. This yields
\begin{equation} \label{vdc reduce 2} \left\vert\left\vert \cavg{g} a_{P(g;h_1)} \right\vert\right\vert \ \leq \ \sqrt{\cavg{h_2} \left\vert\left\vert \cavg{g} a_{P(g;h_1,h_2)} \right\vert\right\vert}.
\end{equation}
Applying \eqref{vdc reduce 2} to the norm in \eqref{vdc reduce}, we find
\[ \left\vert\left\vert \cavg{g} a_{P(g)} \right\vert\right\vert^2 \ \leq \ \cavg{h_1} \sqrt{\cavg{h_2} \left\vert\left\vert \cavg{g} a_{P(g;h_1,h_2)} \right\vert\right\vert},\]
where $P(g;h_1,h_2) = P(g+h_2;h_1)-P(g;h_1)$ is a polynomial in $g$ of degree at most $d-2$ for each $h_1, h_2 \in \mathbb{F}_q$. Proceeding recursively in this way, we see that
\begin{equation}\label{nested roots} \left\vert\left\vert \cavg{g} a_{P(g)} \right\vert\right\vert^2 \ \leq \ \cavg{h_1} \sqrt{\cavg{h_2} \sqrt{ \cdots \sqrt{\cavg{h_{d-2}} \left\vert\left\vert \cavg{g} a_{P(g;h_1,h_2,\ldots,h_{d-2})} \right\vert\right\vert} } }, \end{equation}
where $P(g;h_1,h_2,\ldots,h_{d-2})$ is the polynomial in $g$ of degree at most 2 obtained by differencing. If we reduce the degree one more time---but without applying Cauchy--Schwarz---we will obtain an expression that we can finally bound. Indeed, by repeating the argument of \eqref{vdc reduce}, stopping just before the inequality, observe that for any $h_1,h_2,\ldots,h_{d-2} \in \mathbb{F}_q$, we have
\begin{equation} \label{innermost state} 
\left\vert\left\vert \cavg{g} a_{P(g;h_1,\ldots,h_{d-2})} \right\vert\right\vert^2 \ = \ \cavg{h_{d-1}} \lr{\cavg{g} a_{P(g;h_1,h_2,\ldots,h_{d-1})},a_0}.
\end{equation}
By applying \eqref{innermost state} to the innermost part of \eqref{nested roots}, we have
\begin{equation}\label{final stage} \left\vert\left\vert \cavg{g} a_{P(g)} \right\vert\right\vert^2 \ \leq \ \cavg{h_1} \sqrt{ \cdots \sqrt{\cavg{h_{d-1}} \lr{\cavg{g} a_{P(g;h_1,h_2,\ldots,h_{d-1})},a_0}} }. \end{equation}
After taking square roots on both sides of \eqref{final stage}, we have an expression with $d-1$ radical symbols in it, which is annoying to look at and, fortunately, possible to adjust. By a different form of Cauchy--Schwarz,\footnote{Namely, if $c_1, \ldots, c_q$ are positive numbers, then $\frac{1}{q}\sum_{i=1}^{q} \sqrt{c_i} \leq \sqrt{\frac{1}{q}\sum_{i=1}^q c_i}$. To see this, set $a_j = \sqrt{c_j}$ and $b_k = 1$ in \eqref{csorig}, divide both sides by $q^2$, then take the square root of both sides. For the interested reader, this inequality is first applied with ``$i = h_{d-2}$'' and $c_i = \cavg{h_{d-1}} \lr{\cavg{g} a_{P(g;h_1,h_2,\ldots,h_{d-1})},a_0}$, where we again only abuse notation when clarifying something related to Cauchy--Schwarz.} we can move each of the square roots to the outside of the expression. This massaging of radical symbols is not necessary for the argument, but we hope the reader will appreciate the service. The result is that
\begin{equation} \label{finalest stage} \left\vert\left\vert \cavg{g} a_{P(g)} \right\vert\right\vert \ \leq \ \left( \frac{1}{q^{d-1}} \sum_{h_1,\ldots,h_{d-1}\in \mathbb{F}_q} \lr{\cavg{g} a_{P(g;h_1,h_2,\ldots,h_{d-1})},a_0} \right)^{2^{-(d-1)}}.
\end{equation}
Compare \eqref{finalest stage} and \eqref{test 42}. To finish the proof of the lemma, the idea is that the inner product in \eqref{finalest stage} is zero most of the time, i.e., for enough choices of the parameters $h_i$, and if it's not zero, then it can be bounded trivially.

First, the popular case. Tracing what happens when one differences $P(g) = cg^d$ the whole $d-1$ times, one sees that the linear coefficient of $P(g;h_1,\ldots,h_{d-1})$, viewed as a linear polynomial in $g$, is $d!c\prod_{i=1}^{d-1} h_i$. Since $d < \mathrm{char}(\mathbb{F}_q)$ and $c \neq 0$, this coefficient is zero if and only if at least one $h_i = 0$. Thus, for fixed nonzero $h_1,h_2,\ldots,h_{d-1} \in \mathbb{F}_q^\times$, this linear coefficient is invertible, which implies that the map $g \mapsto P(g;h_1,\ldots,h_{d-1})$ is a permutation of $\mathbb{F}_q$. Using this permutation to reindex a sum, we see that
\begin{equation*}
\lr{\cavg{g} a_{P(g;h_1,h_2,\ldots,h_{d-1})},a_0} \ = \ \lr{\cavg{g} a_g,a_0},
\end{equation*}
and a straightforward calculation shows that the function
\begin{equation*}
\cavg g a_g(x) \ = \ \cavg g 1_A(x + g) -\mu(A)
\end{equation*}
is the zero function, so that the inner product of interest is 0.
Now, for the unpopular case, if any $h_i$ is 0, then $P(g;h_1,\ldots,h_{d-1})$ is a constant polynomial in $g$, and after analyzing the differencing process, one can see that, in particular, the constant is 0. It follows that
\begin{equation*}
\lr{\cavg{g} a_{P(g;h_1,h_2,\ldots,h_{d-1})},a_0} \ = \ \lr{a_{0},a_0} \ = \ ||a_0||^2 \ \leq \ 1.
\end{equation*}
The final step is to reckon the exact popularity of the two cases. Of the $q^{d-1}$ choices of parameters $(h_1,\ldots, h_{d-1})$, exactly $(q-1)^{d-1}$ of them have no $h_i$ equal to zero. Thus, the unpopular case, comprising everything else, happens $q^{d-1}-(q-1)^{d-1}$ times in the sum in \eqref{finalest stage}. It follows that
\begin{equation*}
 \left\vert\left\vert \cavg{g} a_{P(g)} \right\vert\right\vert \ \leq \ \left( \frac{q^{d-1}-(q-1)^{d-1}}{q^{d-1}} \right)^{2^{-(d-1)}}.
\end{equation*}
Thus, take $\mathcal{E}(q,d) = \left( \frac{q^{d-1}-(q-1)^{d-1}}{q^{d-1}} \right)^{2^{-(d-1)}}$. For fixed $d$, $\lim_{q\to\infty} \mathcal{E}(q,d) = 0$.
\end{proof}
Recalling the remark below Theorem~\ref{elem thm}, we observe that Theorem~\ref{elem thm} solves Problem~\ref{two elts; with coeffs} except for those nonprime finite fields with characteristic drawn from the set of primes less than or equal to $d$. For context, this solves more of Problem~\ref{two elts; with coeffs} than Hurwitz's Theorem~\ref{hurwitz thm} does, and less of Problem~\ref{two elts; with coeffs} than Theorem~\ref{small thm} does. We judge this to be a success of elementary methods. Indeed, a method as soft as this cannot be expected to compete directly with careful estimates of numbers of solutions to diagonal equations. On the other hand, the full details of this method fit in these pages easily, and demonstrate the staying power of Cauchy--Schwarz, of differencing, and of averaging. 

\section{A Fourier-analytic approach.}
The first approach was animated by an all-encompassing desire to reduce the degree of a polynomial until it is linear, because a linear polynomial behaves predictably with respect to the averages we have considered---compare the popular and unpopular cases above. The basic structure of the second approach is similar overall: an auxiliary lemma will do most of the work with the help of the same averaging argument. However, in contrast to the erstwhile insistence on differencing, the second approach tolerates polynomiality until it becomes a lightning rod. We will find under this electric influence that the second approach solves Problem~\ref{two elts; with coeffs} in full generality.

\subsection{Preliminaries.}
We need just a little Fourier analysis.

An \emph{additive character} of $\mathbb{F}_q$ is a homomorphism $\chi : (\mathbb{F}_q,+) \to \mathbb{C}^\times$. If $\mathbb{F}_q$ has characteristic $p$, then evidently $\chi(g)^p = \chi(pg) = \chi(0) = 1$ for any $g \in \mathbb{F}_q$, so $\chi$ actually takes values in the set of $p$th roots of unity. As a result, we conclude that $\chi(-g) = \chi(g)^{-1} = \overline{\chi(g)}$. Of course, $\chi$ is an additive character if and only if $\overline{\chi}$ is. The \emph{principal character} $\chi_0$ is identically $1$. The set $\widehat{\mathbb{F}_q}$ of additive characters of $\mathbb{F}_q$ is an orthonormal basis for the $q$-dimensional $\mathbb{C}$-linear space of functions $f : \mathbb{F}_q \to \mathbb{C}$ with the inner product $\lr{\cdot,\cdot}$ defined in the previous section. Hence, with a small amount of work on orthogonality of characters, one can show that any $f$ can be written as a linear combination of additive characters in the following way:
\begin{equation}\label{basisdecomp}
f \ = \ \sum_{\chi \in \widehat{\mathbb{F}_q}} \lr{f,\chi}\chi.
\end{equation} The \emph{Fourier transform} of $f$, written $\hat{f}$, is the function $\hat{f} : \widehat{\mathbb{F}_q} \to \mathbb{C}$ given by $\hat{f}(\chi) = q\lr{f,\overline{\chi}} = \sum_{g \in \mathbb{F}_q} f(g)\chi(g)$. Thus $\lr{f,\chi} = \frac{\hat{f}(\overline{\chi})}{q}$; plugging this into~\eqref{basisdecomp} and reindexing the sum yields the Fourier inversion formula, valid for all $a \in \mathbb{F}_q$:
\begin{equation}\label{inversion formula} f(a) \ = \ \sum_{\chi \in \widehat{\mathbb{F}_q}} \frac{\hat{f}(\overline{\chi})}{q} \chi(a) \ = \ \frac{1}{q} \sum_{\chi \in \widehat{\mathbb{F}_q}} \hat{f}(\chi) \chi(-a).
\end{equation}
With some more work, one can show the Plancherel formula: For any $f_1, f_2 : \mathbb{F}_q \to \mathbb{C}$, we have $\lr{\hat{f_1},\hat{f_2}} = q \lr{f_1,f_2}$, where on the left-hand side we mean the analogous inner product
\begin{equation*} \lr{\hat{f_1},\hat{f_2}} := \frac 1q \sum_{\chi \in \widehat{\mathbb{F}_q}} \hat{f_1}(\chi) \overline{\hat{f_2}(\chi)}.
\end{equation*}
Fourier analysis is a huge topic; for a sample, the interested reader may consult \cite{hr, korner, ss}. 
\subsection{The content.} 
We will state a variant of Lemma~\ref{eqd}, use it to solve Problem~\ref{two elts; with coeffs}, and then prove it.
\begin{lemma}\label{eqdver2} There is a positive function $\mathcal{E}(q,d) : \mathbb{N}^2 \to \mathbb{R}$ with these properties:
\begin{enumerate}
\item For any $q$, any subsets $A_q, B_q \subseteq \mathbb{F}_q$, and any polynomial $P \in \mathbb{F}_q[x]$ with degree $d$ coprime to $q$, we have
\begin{equation*} \left\vert \frac 1q \sum_{g \in \mathbb{F}_q} \mu(A_q \cap (B_q + P(g))) - \mu(A_q)\mu(B_q) \right| \ \leq \ \mathcal{E}(q,d).
\end{equation*}
\item For any fixed $d$, we have $\lim_{q\to\infty} \mathcal{E}(q,d) \ = \ 0$. 
\end{enumerate}
\end{lemma}
\begin{remarks} These comments are parallel to their corresponding comments under Lemma~\ref{eqd}.
\begin{enumerate}
\item This lemma strengthens Lemma~\ref{eqd}, since the set of polynomials appearing in the first statement contains the set of polynomials with degree between 1 and $\mathrm{char}(\mathbb{F}_q)$.
\item The intuitive assertion of this lemma remains the same as in Lemma~\ref{eqd}; namely, the quantities $\mu(A_q \cap (B_q + P(g)))$ are eventually of size $\mu(A_q)\mu(B_q)$ on average. Moreover, the application of this lemma will be almost exactly the same.
\item Our choice of $\mathcal{E}(q,d)$ here will decay to zero more quickly than it did in Lemma~\ref{eqd}, except when $d = 2$, in which case it agrees with the previous choice.
\end{enumerate}
\end{remarks}
\begin{proof}[Solution to Problem~\ref{two elts; with coeffs}] We may suppose $q > d$. For $q$ such that $d$ is coprime with $q$, the argument proceeds as in the proof of Theorem~\ref{elem thm}, with Lemma~\ref{eqdver2} replacing Lemma~\ref{eqd}, and shows that for all sufficiently large $q$ coprime to $d$, for all $a,b \in \mathbb{F}_q^\times$, we have
\begin{equation}\label{finish it} \mathbb{F}_q \ = \ \{ ax^d + by^d : x,y\in\mathbb{F}_q\}.
\end{equation}
Otherwise, namely, if $\gcd (d,q) > 1$, then by Lemma~\ref{bezoutarg}, we have
\begin{equation*} \{ ax^d + by^d : x,y \in \mathbb{F}_q \} \ = \ \{ ax^\delta + by^\delta : x,y\in\mathbb{F}_q\},
\end{equation*} 
where $\delta = \gcd (d, q-1)$. Now, $\delta$ divides $q-1$, so $\delta$ and $q$ are coprime. This completes the argument since $d$ has finitely many factors.
\end{proof}

\begin{proof}[Proof of Lemma~\ref{eqdver2}]
Fix $q$, subsets $A,B \subseteq \mathbb{F}_q$, and a polynomial $P \in \mathbb{F}_q[x]$ with degree $d$ coprime to $q$.

We rewrite the expression
\begin{equation*} \cavg g \mu(A \cap (B + P(g)))
\end{equation*}
step by step. Recall that $\mu(S) = \frac{|S|}{q} = \cavg h 1_S(h)$, that $1_{S \cap S'}(x) = 1_{S}(x)1_{S'}(x)$, and that $1_{S+y}(x) = 1_S(x-y)$ for any $x,y \in \mathbb{F}_q$ and $S,S' \subseteq \mathbb{F}_q$. Thus,
\begin{equation*} \cavg g \mu(A \cap(B+P(g))) \ = \ \cavg g \cavg h 1_A(h)1_{B}(h-P(g)).
\end{equation*}
Inverting using \eqref{inversion formula} with $f = 1_B$ and $a = h-P(g)$, we find
\begin{equation*}
\cavg g \cavg h 1_A(h)1_{B}(h-P(g)) \ = \ \frac{1}{q^3} \sum_{g,h \in \mathbb{F}_q} \sum_{\chi \in \widehat{\mathbb{F}_q}} 1_A(h) \widehat{1_B}(\chi) \chi(P(g)-h).
\end{equation*}
After separating $\chi(P(g)-h) = \chi(P(g))\chi(-h) = \chi(P(g))\overline{\chi}(h)$ and changing the order of summation, we notice the expression for the Fourier transform of $1_A$ evaluated at the character $\overline{\chi}$:
\begin{equation*}
\frac{1}{q^3} \sum_{g,h \in \mathbb{F}_q} \sum_{\chi \in \widehat{\mathbb{F}_q}} 1_A(h) \widehat{1_B}(\chi) \chi(P(g)-h) \ = \ \frac{1}{q^3} \sum_{g, \chi} \chi(P(g))\widehat{1_B}(\chi) \underbrace{\sum_{h \in \mathbb{F}_q} 1_A(h)\overline{\chi}(h)}_{= \widehat{1_A}(\overline{\chi}) }. 
\end{equation*}
After changing the order of summation again, we have
\begin{equation*}
\frac{1}{q^3} \sum_{g, \chi} \chi(P(g))\widehat{1_B}(\chi) \widehat{1_A}(\overline{\chi}) \ = \ \frac{1}{q^3} \sum_\chi \widehat{1_A}(\overline{\chi}) \widehat{1_B}(\chi) \sum_{g} \chi(P(g)).
\end{equation*}
Let's separate the $\chi = \chi_0$ term from the whole sum. Remember, $\chi_0$ is identically 1. Thus, by definition, $\widehat{1_A}(\chi_0) = \sum_{g \in \mathbb{F}_q} 1_A(g) \chi_0(g) =  \sum_{g \in \mathbb{F}_q} 1_A(g) = |A|$, and since $\overline{\chi_0} = \chi_0$, we have $\widehat{1_B}(\overline{\chi}) = |B|$. Moreover, $\sum_{g \in \mathbb{F}_q} \chi_0(P(g)) = q$. Now $\frac{1}{q^3} \cdot |A||B|q = \mu(A)\mu(B)$. Thus
\begin{equation*}
\frac{1}{q^3} \sum_\chi \widehat{1_A}(\overline{\chi}) \widehat{1_B}(\chi) \sum_{g} \chi(P(g)) \ = \ \mu(A)\mu(B) + \frac{1}{q^3} \sum_{\chi \neq \chi_0} \widehat{1_A}(\overline{\chi}) \widehat{1_B}(\chi) \sum_{g} \chi(P(g)).
\end{equation*}

Thus, to prove the lemma, we are looking for a function $\mathcal{E}(q,d)$ to bound
\begin{equation}\label{thing}
\left\vert \frac{1}{q^3} \sum_{\chi \neq \chi_0} \widehat{1_A}(\overline{\chi}) \widehat{1_B}(\chi) \sum_{g \in \mathbb{F}_q} \chi(P(g)) \right\vert.
\end{equation}
By the triangle inequality, \eqref{thing} is bounded by
\begin{equation}\label{step 1}
\frac{1}{q^3} \sum_{\chi \neq \chi_0} \left\vert \widehat{1_A}(\overline{\chi}) \right\vert \left\vert \widehat{1_B}(\chi) \right\vert \left\vert \sum_{g \in \mathbb{F}_q} \chi(P(g))\right\vert.
\end{equation}
At this point, we await the lightning. Transcribing some folklore spawned from his own work, Weil wrote down in a short note in 1948 the relationship between some exponential sums and the Riemann hypothesis over finite fields \cite{weil48}. The specific formulation we use here appears\footnote{For an offline source, see \cite[Theorem 2E]{schmidtold} or \cite[Theorem 2.5]{schmidt}.} (with minor notation adjustments) in Kowalski's notes on exponential sums \cite[Theorem 3.2]{kow}:
\begin{theorem} Fix a polynomial $f \in \mathbb{F}_q[x]$ of degree $d$ and a nontrivial additive character $\chi$ of $\mathbb{F}_q$. If $d < q$ and $d$ is coprime to $q$, then
\begin{equation*} \left\vert \sum_{g \in \mathbb{F}_q} \chi(f(g)) \right\vert \ \leq \ (d-1)\sqrt{q}.
\end{equation*}
\end{theorem}
\noindent This theorem applies readily to part of \eqref{step 1}. Thus we find that \eqref{step 1} is bounded by
\begin{equation}\label{step 2}
\frac{(d-1)\sqrt{q}}{q^3} \sum_{\chi \neq \chi_0} \left\vert \widehat{1_A}(\overline{\chi}) \right\vert \left\vert \widehat{1_B}(\chi) \right\vert.
\end{equation}

Our final step uses the Plancherel formula. In preparation, we massage part of \eqref{step 2} by padding it with the $\chi = \chi_0$ term and applying Cauchy--Schwarz to see that
\begin{align}
\sum_{\chi \neq \chi_0} \left\vert \widehat{1_A}(\overline{\chi}) \right\vert \left\vert \widehat{1_B}(\chi) \right\vert \ & \leq \ \sum_{\chi} \left\vert \widehat{1_A}(\overline{\chi}) \right\vert \left\vert \widehat{1_B}(\chi) \right\vert \nonumber \\
& \leq \ \left( \sum_{\chi} \left\vert \widehat{1_A}(\chi) \right\vert^2 \sum_{\chi'} \left\vert \widehat{1_B}(\chi') \right\vert^2 \right)^\frac{1}{2} \nonumber \\
& = \ \left( q^4\mu(A)\mu(B) \right)^\frac{1}{2} \label{cheeky}\\
& \leq \ q^2, \nonumber
\end{align}
where \eqref{cheeky} holds by Plancherel since
\begin{equation*}
\sum_{\chi} \left\vert \widehat{1_A}(\chi) \right\vert^2 \ = \ q\lr{\widehat{1_A},\widehat{1_A}} \ = \ q^2 \lr{1_A,1_A} \ = \ q^2\mu(A).
\end{equation*}
Thus \eqref{step 2} is bounded by
\begin{equation*}
\frac{(d-1)\sqrt{q}}{q^3}\cdot q^2 \ = \ \frac{d-1}{\sqrt{q}},
\end{equation*}
so take $\mathcal{E}(q,d) = \frac{d-1}{\sqrt{q}}$. Obviously $\lim_{q\to\infty} \mathcal{E}(q,d) = 0$ for fixed $d$.
\end{proof}

In this approach we depend on serious, now classical, knowledge about certain exponential sums, a definite step up in difficulty over Gauss sums. The interest in these kinds of exponential sums arose in part as an outgrowth of interest in the problems of cyclotomy and diagonal equations but has since taken on its own life in number theory. We have certainly not done justice to the Riemann hypothesis over finite fields here; it is a deep topic. For a historical perspective, see \cite{roq}, and for a mathematical discussion see, for example, \cite{IR}. What we have done is examined several facets of a pretty problem, Problem~\ref{two elts; with coeffs}, found in the same deposit as Fermat's last theorem.

\section*{Acknowledgment.}
We thank the referees for many useful comments and especially the referee who suggested several improvements to the structure of Section~\ref{sec 2} and pointed out what are now Proposition 2, the remark under Theorem~\ref{elem thm}, and part of the first remark under Lemma~\ref{eqd}.

%
%
%
%
\vfill\eject

\end{document}